\documentclass[11pt,a4paper]{amsart} 

\usepackage{latexsym}

\usepackage[a4paper , lmargin = {2.5cm} , rmargin = {2.5cm} , tmargin = {2.5cm} , bmargin = {2.5cm} ]{geometry} 

\usepackage{geometry}
\usepackage{pinlabel}         
\usepackage{graphicx}
\usepackage{amssymb, amsthm}
\usepackage{epstopdf}
\usepackage{romannum}
\usepackage{tikz}
\usepackage{subfig}
\usepackage{MnSymbol}
\usepackage[arrow, matrix, curve]{xy} %zum Erstellen von kommutativen Diagrammen

        \newcommand{\N}{\ensuremath{\mathbb{N}}}  
  	\newcommand{\Z}{\ensuremath{\mathbb{Z}}}   
  	   
  	\newcommand{\R}{\ensuremath{\mathbb{R}}} 

   	\def\Aut{{\rm{Aut}}$(F_n)$}
        \def\Autn{{\rm{Aut}}(F_n)}
        
   	\def\SAut{{\rm{SAut}}$(F_n)$}
        \def\SAutn{{\rm{SAut}}(F_n)}
	\def\GL{{\rm{GL}}}

	\def\W{{\rm{W}}}
	\def\Eig{{\rm{Eig}}}
	\def\N{{\rm{N}}}
 	\def\SW{{\rm{SW}}}
	\def\Z{{\mathbb{Z}}}
	
	\def\SN{{\rm{SN}}}
	
	\def\im{{\rm{im}}}

\theoremstyle{plain}
\newtheorem*{NewTheorem}{Theorem}

\newtheorem{theorem}{Theorem}[section]
\newtheorem{lemma}[theorem]{Lemma}

\newtheorem{proposition}[theorem]{Proposition}

\newtheorem{remark}[theorem]{Remark}

\title[Low dimensional linear representations of \SAut]{Low dimensional linear representations of $\mathbf{SAut}\boldsymbol{(F_n)}$}
\author{Olga Varghese}
\date{\today}
\address{Olga Varghese\\
Department of Mathematics\\
M\"unster University\\ 
Einsteinstra\ss e 62\\
48149 M\"unster (Germany)}
\email{olga.varghese@uni-muenster.de}

\begin{document}
\pagenumbering{arabic}
\begin{abstract}
We prove that \SAut, the unique subgroup of index two in the automorphism group of a free group of rank $n$, admits no non-trivial linear representation of degree $d<n$ for any field of characteristic not equal to two.
\end{abstract}

\maketitle

\section{Introduction}

In this work we study low dimensional linear representations of the group \SAut\ and we prove strong rigidity results for these.  To be precise, let $\mathbb{Z}^n$ be the free abelian group and $F_{n}$ the free group of rank $n$. 
The abelianization map $F_{n}\twoheadrightarrow \Z^{n}$ gives a natural epimorphism 
$\Autn\twoheadrightarrow\GL_{n}(\Z).$
The group \SAut\ is defined as the preimage of ${\rm SL}_{n}(\Z)$ under this map. The group ${\rm SL}_{n}(\Z)$  acts faithfully by linear transformations on $\mathbb{R}^n$. 
We shall prove that the linear representation: 
$\SAutn\twoheadrightarrow{\rm SL}_{n}(\Z)\hookrightarrow{\rm SL}_n(\R)$
is minimal in the following sense:
\begin{NewTheorem}
Let $n\geq 3$ and let $\rho:\SAutn\rightarrow{\rm SL}_{d}(K)$
be a linear representation of degree $d$ over a field $K$ with ${\rm char}(K)\neq 2$. If $d<n$, then $\rho$ is trivial.
\end{NewTheorem}

We note that the group \SAut\ is perfect, therefore the image of a linear representation of \SAut\ is a subgroup of ${\rm SL}_{d}(K)$. 

Indeed, \textsc{Bridson} and \textsc{Vogtmann} proved in \cite{Vogtmann} that if $n\geq 3$ and $d<n$, then \SAut\ cannot act non-trivially by homeomorphisms on any contractible manifold of dimension $d$. Using their techniques we proved in a purely group theoretical way the above theorem.

\section{Linear representations}
\subsection{The automorphism group of a free group}

As the main protagonist in this work is the group \SAut, we start with the definition of this group, and establish some notation to be used throughout. 

We begin with the definition of the automorphism group of the free group of rank~$n$.
Let $F_{n}$ be the free group of rank $n$ with a fixed basis $X:=\left\{x_{1}, \ldots, x_{n}\right\}$. We denote by \Aut\ the automorphism group of $F_n$ and by \SAut\ the unique subgroup of index two in \Aut. 

Let us first introduce a notations for some elements of \Aut.   We define  involutions $(x_{i},x_{j})$ and $e_{i}$ for $i, j=1,\ldots, n$,~$i\neq j$ as follows:
\[
\begin{matrix}
(x_{i},x_{j})(x_{k}):=\begin{cases} x_{j} & \mbox{if $k=i$,}  \\ 
			     x_{i} & \mbox{if $k=j$,}  \\ 
			     x_{k} & \mbox{if $k\neq i,$ $j$.}  
\end{cases}
&
e_{i}(x_{k}):=\begin{cases} x^{-1}_{i} & \mbox{if $k=i$,}  \\
				 x_{k} & \mbox{if $k\neq i$}.   
\end{cases}
\end{matrix}
\]

\subsection[Some finite subgroups of \Aut]{Some finite subgroups of $\mathbf{Aut}\boldsymbol{(F_n)}$}
We begin this subsection by describing some finite subgroups of \Aut\ and \SAut:
\begin{align*}
\Sigma_{n}&:=\left\{\alpha\in{\rm Aut}(F_{n})\mid\alpha_{|X}\in{\rm Sym}(X)\right\}\cong{\rm Sym}(n),\\
     \N_{n}&:=\langle\left\{e_{i} \mid i=1,\ldots, n\right\}\rangle\cong\Z^{n}_{2},\\
     \SN_{n}&:=\N_{n}\cap \SAutn=\langle\left\{e_{1}e_i\mid i=2,\ldots,n\right\}\rangle\cong\Z^{n-1}_{2},\\
     \W_{n}&:=\langle \N_{n}\cup\Sigma_{n}\rangle=\N_{n}\rtimes\Sigma_{n},\\
     \SW_{n}&:= \W_{n}\cap \SAutn.
\end{align*}

The following variant of a result by \textsc{Bridson} and \textsc{Vogtmann} \cite[3.1]{Vogtmann} will be used here to prove that certain actions on spaces of \SAut\ are indeed trivial. For a detailed proof the reader is referred to \cite[1.13]{Diplomarbeit}.
\begin{proposition}
\label{faktorisiertSL}
Let $n\geq 3$, $G$ be a group and $\phi : \SAutn\rightarrow G$ a group homomorphism. 
\begin{enumerate}
\item[$(i)$] If  there exists $\alpha\in\SW_n-\left\{{\rm id}_{F_n}, e_{1}e_{2}\ldots e_{n}\right\}$ with $\phi(\alpha)=1$, then  $\phi$ factors through ${\rm SL}_{n}(\Z_{2})$.
\item[$(ii)$] If $n$ is even and  $\phi(e_{1}e_{2}\ldots e_{n})=1$, then $\phi$ factors through ${\rm PSL}_{n}(\Z)$.
\item[$(iii)$] If there exists $\alpha\in\SW_{n}-\SN_{n}$ with $\phi(\alpha)=1$, then $\phi$ is trivial.
\end{enumerate}
\end{proposition}

We divide the proof of our theorem into two steps. In the first step we show that for $d<n$ 
any homomorphism $\widetilde{\rho}: {\rm SL}_{n}(\Z_{2})\rightarrow{\rm SL}_d(K)$ is trivial. In the second step we prove by induction on~$n$ that for $d<n$ any homomorphism $\rho:\SAutn\rightarrow{\rm SL}_{d}(K)$ factors through ${\rm SL}_{n}(\Z_{2})$.
\[
\begin{xy}
\xymatrix
{
{\rm SAut}(F_{n}) \ar[rr]^{\rho} \ar[dr]_{\pi} & & {\rm SL}_d(K) \\
& {\rm SL}_{n}(\Z_2) \ar[ur]_{\widetilde{\rho}} 
}
\end{xy}
\]

A key observation is the following proposition.
\begin{proposition}
\label{UntergruppeTrivial}
Let $K$ be a field with ${\rm char}(K)\neq 2$ and let $\phi: \Z^{m}_{2}\rightarrow{\rm SL}_{d}(K)$ be a group homomorphism. If $d\leq m$, then $\phi$ is not injective.
\end{proposition}
\begin{proof}
First, we recall an important characterization of diagonalizable linear maps: a linear map is diagonalizable over the field $K$ if and only if its minimal polynomial is a product of distinct linear factors over $K$, see \cite[Thm. 6, p.~204]{Hoffmann}. The minimal  polynom  of an involution is a divider of $x^2-1=(x+1)(x-1)$ and satisfies this characterization if ${\rm char}(K)\neq 2$. We observe that all elements in the image of $\phi$ have order less or equal to two, therefore the elements in the image of $\phi$ are diagonalizable. 

Next, we note that all elements in the image of $\phi$ commute, therefore these are simultaneously diagonalizable, see \cite[p.~51-53]{Horn}.

For $A$ in the image of $\phi$ we have $\Eig(A,1)\oplus\Eig(A,-1)=K^{d}$, where we denote by $\Eig(A,1)$ resp. $\Eig(A,-1)$ the eigenspace for the eigenvalue $1$ resp. $-1$. We note that $\det(A)=1$. 
The number of diagonal matrices in ${\rm SL}_d(K)$ with an even number of entries equal to $-1$ is given by
\[
\sum\limits_{i=0}^{\lfloor\frac{d}{2}\rfloor}\binom{d}{2i}=2^{d-1}.
\]
Therefore the image of $\phi$ contains at most $2^{d-1}$ elements. But we have  
\[
{\rm ord}(\Z^{m}_{2})=2^m>2^{d-1}\geq {\rm ord}(\im(\phi(\Z^{m}_{2})))
\]
and therefore $\phi$ is not injective.
\end{proof}

Using Proposition \ref{UntergruppeTrivial} we obtain the following result. 
\begin{proposition}
\label{SLTrivial}
Let $n\geq 3$ and $\widetilde{\rho}:{\rm SL}_{n}(\Z_{2})\rightarrow{\rm SL}_d(K)$
be a linear representation of degree $d$ over a field $K$ with ${\rm char}(K)\neq 2$.
If $d<n$, then $\widetilde{\rho}$ is trivial.
\end{proposition}
\begin{proof}
By simplicity of ${\rm SL}_n(\Z_{2})$ the kernel of the map $\widetilde{\rho}$ is either trivial or all of~${\rm SL}_n(\Z_{2})$.
The group ${\rm SL}_{n}(\Z_{2})$ contains a subgroup $U$ which is isomorphic to~$\Z^{n-1}_{2}$,
$U=\langle\left\{{\rm E}_{12},\ldots, {\rm E}_{1n}\right\}\rangle$,
where we denote by ${\rm E}_{1i}$ the matrix  which has ones on the main diagonal and in the entry $(1,i)$ and zeros elsewhere.
By Proposition \ref{UntergruppeTrivial} the restriction of  $\widetilde{\rho}$ to this group is not injective, therefore the kernel of the map $\widetilde{\rho}$ is equal to ${\rm SL}_n(\Z_{2})$.
\end{proof}

First, we show the main theorem for $n=3$ and $n=4$ and then proceed by induction on~$n$.
\begin{lemma}
\label{n3}
Let $\rho:{\rm SAut}(F_{3})\rightarrow{\rm SL}_d(K)$
be a linear representation of degree $d$ over a field $K$ with ${\rm char}(K)\neq 2$. If $d<3$, then $\rho$ is trivial.
\end{lemma}
\begin{proof}
The group ${\rm SN}_{3}$ is isomorphic to $\Z_{2}\times\Z_{2}$ and by Proposition~\ref{UntergruppeTrivial}  it follows that the restriction of $\rho$ to this group is not injective. The element $e_1e_2e_3$ is not in~${\rm SN}_3$ and therefore by Proposition~\ref{faktorisiertSL} the map
$\rho$ factors through ${\rm SL}_3(\Z_2)$. Using Proposition \ref{SLTrivial} it follows that $\rho$ is trivial.
\end{proof}
\begin{lemma}
\label{n4}
Let $\rho:{\rm SAut}(F_{4})\rightarrow{\rm SL}_d(K)$ be a linear representation of degree $d$ with ${\rm char}(K)\neq 2$. If $d<4$, then $\rho$  is trivial.
\end{lemma}
\begin{proof}
The subgroup ${\rm SN}_{4}$ is isomorphic to $\Z^{3}_{2}$ and by Proposition \ref{UntergruppeTrivial} it follows that~$\rho_{|{\rm SN}_{4}}$ is not injective. Therefore there exists an element
$\alpha\in{\rm SN}_{4}-\left\{{\rm id}_{F_{4}}\right\}$ with~$\rho(\alpha)={\rm I}_d$, where ${\rm I}_d$ is the unit matrix. 
If $\alpha$ is not equal to $e_1e_2e_3e_4$, then by Proposition \ref{faktorisiertSL} it follows that $\rho$ factors through ${\rm SL}_{4}(\Z_{2})$ and the triviality of $\rho$ follows by Proposition \ref{SLTrivial}.
If $\alpha$ is equal to $e_1e_2e_3e_4$ then by Proposition~\ref{faktorisiertSL} the map $\rho$ factors through ${\rm PSL}_{4}(\Z)$.
\[
\begin{xy}
\xymatrix
{
{\rm SAut}(F_{4}) \ar[rr]^{\rho} \ar[dr]_{\pi} & & {\rm SL}_d(K) \\
& {\rm PSL}_{4}(\Z) \ar[ur]_{\widetilde{\rho}} 
}
\end{xy}
\]
The group ${\rm PSL}_{4}(\Z)$ contains a subgroup $U$ which is isomorphic to $\Z^{4}_{2}$, namely   
\[
U=\langle\left\{[{\rm E}_{-1}\rm{E}_{-2}], [\rm{E}_{-2}\rm{E}_{-3}], [\rm{P}_{12}\rm{P}_{34}], [\rm{P}_{13}\rm{P}_{24}]\right\}\rangle,
\]
where we denote by ${\rm E}_{-i}$ the matrix  which has ones on the main diagonal except the entry $(i,i)$ and zeros elsewhere and the entry $(i,i)$ is equal to $-1$. The matrix ${\rm P}_{ij}$ is a permutation matrix. 

By Proposition \ref{UntergruppeTrivial} it follows that there exists an element $[A]\in U-\left\{[\rm{I}_4]\right\}$ with~$\widetilde{\rho}([A])={\rm I}_d$. We consider a preimage of $[A]$ under $\pi$ and note that there exists an element $\beta\in{\rm SW}_{n}-\left\{{\rm id}_{F_{n}},e_1e_2e_3e_4\right\}$  with~$\rho(\beta)=\widetilde{\rho}\circ\pi(\beta)=\widetilde{\rho}([A])={\rm I}_d$.
By Proposition \ref{faktorisiertSL} it follows that $\rho$ factors through ${\rm SL}_{4}(\Z_{2})$ and by Proposition \ref{SLTrivial} we obtain triviality of $\rho$.
\end{proof}
$ $

\textsl{Proof of main theorem.}\\
We proceed by induction on $n$.
We assume that $n>4$. 
If $\rho(e_1e_2)={\rm I}_d$, then by Proposition \ref{faktorisiertSL} the map $\rho$ factors  through ${\rm SL}_{n}(\Z_{2})$ and by Proposition \ref{SLTrivial} it follows that $\rho$ is trivial. Otherwise $\rho(e_1e_2)$ is a non-trivial involution. We note that the dimension of the eigenspace ${\rm Eig}(\rho(e_1 e_2), 1)$ is equal to  $d-2\cdot l$ for some $l\in\mathbb{N}_{>0}$.
The centralizer of $e_1e_2$, which we will denote by $C(e_1 e_2)$,
contains a subgroup $U$ isomorphic to ${\rm SAut}(F_{n-2})$. More precisely:
let $\left\{x_{1},\ldots,x_{n}\right\}$ be a basis of $F_{n}$ and let $\left\{x_{3},\ldots, x_{n}\right\}$ be a basis of $F_{n-2}$.
We define
\[
U:=\left\{f\in C(e_{1}e_{2})\mid f(x_{1})=x_{1}, f(x_{2})=x_{2} \text{ and} f_{|\left\{x_{3},\ldots, x_{n}\right\}}\in{\rm SAut}(F_{n-2})\right\}.
\]
Then $U$ is isomorphic to ${\rm SAut}(F_{n-2})$. By the induction assumption any homomorphism \[
\rho':{\rm SAut}(F_{n-2})\rightarrow{\rm SL}_{d'}(K)
\]
for $d'<n-2$ is trivial. The restriction of $\rho$ to $U$ acts on ${\rm Eig}(\rho(e_1 e_2), 1)\cong K^{d-2l}$ and therefore
\[
\rho_{|U}: U\rightarrow{\rm SL}_{d-2\cdot l}(K)
\]
is trivial. In particular, the element $e_3 e_4$ is in $U$ and acts trivially on ${\rm Eig}(\rho(e_1 e_2), 1)$. 
It follows that ${\rm Eig}(\rho(e_1 e_2,1)\subseteq{\rm Eig}(\rho(e_3 e_4), 1)$.
This argument is symmetric and we obtain  
\[
{\rm Eig}(\rho(e_1 e_2),1)={\rm Eig}(\rho(e_3 e_4),1).
\] 
But then we have two~commuting  involutions with equal eigenspaces of eigenvalue $1$, therefore $\rho(e_1 e_2)$ and~$\rho(e_3 e_4)$ are equal and the element~$e_1 e_2 e_3 e_4$ acts trivially.
We have $n>4$ and by Proposition~\ref{faktorisiertSL} the map $\rho$ factors through ${\rm SL}_{n}(\Z_2)$. By Proposition~\ref{SLTrivial} it follows that $\rho$ is trivial.

\hfill $\square$\par\medskip

We finish this work with the following remark.

\begin{remark}
The key ingredient in the proof of main theorem is that there exist only finitely many pairwise commuting involutions in ${\rm SL}_d(K)$ when ${\rm char}(K)\neq 2$, see Proposition \ref{UntergruppeTrivial}.
This is no longer true for infinite fields of characteristic $2$, as we can for example consider:
\[
\left\{
\begin{pmatrix}
1 & a & 0\\
0 & 1 & 0\\
0 & 0 & {\rm I}_{d-2}
\end{pmatrix} \Bigg\vert \ a\in K^ * \right\}\subseteq {\rm SL}_d(K).
\]
\end{remark}

\end{document}